\documentclass[12pt]{article}

\usepackage{amsmath}
\usepackage{amsthm}
\usepackage{amsfonts}
\usepackage{amssymb}
\usepackage{amscd}
\usepackage{calc}
\usepackage{graphicx}
\usepackage{xypic}

\newtheorem{thm}{Theorem}[section]
\newtheorem{lemma}[thm]{Lemma}
\newtheorem{prop}[thm]{Proposition}
\newtheorem{cor}[thm]{Corollary}

\newcommand{\Nb}{\mathbb{N}}
\newcommand{\Zb}{\mathbb{Z}}

\newcommand{\OO}{\mathcal{O}}
\newcommand{\PP}{\mathcal{P}}

\newcommand{\BB}{\mathcal{B}}
\newcommand{\FF}{\mathcal{F}}

\newcommand{\Ac}{\mathcal{A}}

\newcommand{\Ker}{\text{Ker}}
\newcommand{\Imm}{\text{Im}}

\newcommand{\GL}{\text{GL}}

\newcommand{\AF}{\overline{\Ac_{\FF_1\cup \FF_2}}}
\newcommand{\BGL}{\BB_e\GL}

\title{Maximal Levi Subgroups Acting on the Euclidean Building of  $\GL_n(F)$}
\author{Jonathan Needleman}
\begin{document}
\maketitle
\begin{abstract}
In this paper we give a complete invariant of the action of $\GL_n(F)\times \GL_m(F)$ on the Euclidean building $\BGL_{n+m}(F)$, where $F$ is a non-archimedian field.  We then use this invariant to give a natural metric on the resulting quotient space.  In the special case of the torus acting on the tree $\BGL_2(F)$ this gives a method for calculating the distance of any vertex to any fixed apartment.
\end{abstract}
\section{Introduction}

\label{secBe}
To understand distance in the 1-skeleton of a building $\BB G$ associated to a reductive algebraic group  $G$, one may look at a stabilizer $K$ of a point, and then study the action of $K$ on $\BB G$.  When working over a non-archimedian field vertices correspond to maximal compact subgroups.  This analysis gives rise to information about $K\backslash G/K$, and  therefore the Hecke algebra \cite{She:1},\cite{She:2}.\\
\\
In this paper we specialize to $G=\GL_n(F)$ and are interested in the double cosets $L\backslash G/K$, where $L\cong \GL_{n_1}(F)\times \GL_{n_2}(F)$ is a maximal Levi subgroup of $G$.  The study of the action of $L$ on the building $\BB_e\GL_n(F)$ will lead to a description of distance from any vertex to a certain subbuilding stabilized by $L$.  In the case when $n=2$ and $L=T$ is a maximal split torus, our description gives a way of calculating the distance from a given point to a fixed apartment.\\
\\
We also give a combinatorial description of the quotient space $L\backslash \BB_e \GL_n(F)$ as follows.  Let $A^n=\{(\alpha_i)_{i=1}^n|\alpha_i\in \Nb, \alpha_I\geq \alpha_{i+1}\}$.  Then if $n_1\leq n_2$ there is an graph isometry between $L\backslash \BB_e \GL_n(F)$ and $A^{n_1}$ when $A^n$ is endowed with the following metric: $d(\alpha,\beta)=\text{max}_{i=1}^n |\alpha_i-\beta_i|$ where $\alpha,\beta\in A^n$.  This result shows that $1$-skeleton of the resulting quotient space only depends on $\text{min}(n_1,n_2)$. 
\\
This paper is broken up into two main sections.  The first gives a description of the building in terms of $\OO$-lattices and describes an invariant of the action of $L$ on this building.  The second section gives a geometric interpretation of this invariant, yielding a combinatorial description of the quotient space $L\backslash \BGL_n(F)$. 
\section{Orbits of Maximal Levi Factors on $\BGL(V)$}
\subsection{$\OO$-Lattices and $\BGL(V)$}
Throughout this paper let $F$ be a non-archimedian field.  We will denote the ring of integers in $F$ by $\OO$, and fix once an for all a uniformizer $\varpi$ of $\OO$.  Let the unique maximal prime ideal be denoted as $\PP=(\varpi)$, and the residue field $\OO/\PP$ of order $p^k=q$ will be denoted by $\mathfrak{k}$.  Let $\PP^k=(\varpi^k)$ for $k\in\Zb$.  Then $\log_\PP(\PP^k)=k$.  Let $V$ be a vector space defined over $F$.  We will describe the Euclidean building $\BB_e\GL(V)$ associated to $GL(V)$.  For more details see \cite{Br} or \cite{Ga}.  Let $\Lambda\subset V$ be a finitely generated free  $\OO$-module.  Denote by $[\Lambda]$ the homothety class of $\Lambda$, that is $[\Lambda]=\{a\Lambda|a\in F^\times\}$.  \\
\\
Homothety classes of lattices will form the vertices of $\BB_e\GL(V)$.  Two vertices $\lambda_1,\lambda_2\in \BB_e\GL(V)$ are incident if there are representatives $\Lambda_i\in \lambda_i$ so that $\varpi\Lambda_1\subset \Lambda_2 \subset \Lambda_1$, i.e. $\Lambda_2/\varpi\Lambda_1$ is a $\mathfrak{k}$-subspace of $\Lambda_1/\varpi\Lambda_1$.  The chambers in $\BB_e\GL(V)$ are collections of maximally incident vertices.  To put this more concretely, assume the dimension of $V$ is $n$.  Then a chamber is a collection of $n$ vertices $\lambda_0 \cdots\lambda_{n-1}$ with representatives $\Lambda_0 \cdots \Lambda_{n-1}$ satisfying $\varpi\Lambda_0\subsetneq \Lambda_1 \subsetneq \cdots \subsetneq \Lambda_{n-1}\subsetneq \Lambda_0$.  A wall of a chamber is any subset of $n-1$ vertices in the given chamber.  We will denote by $\BB_e\GL(V)^k$ the set of all facets of $\BB_e\GL(V)$ of dimension $k$.\\
\\
A frame $\FF$ in $V$ is a collection of lines $l_1,\ldots l_n\subset V$ which are linearly independent and span all of $V$.  We now describe certain subcomplexes of $\BB_e\GL(V)$.  Define $\Ac_\FF$ to be the subcomplex consisting of vertices $[\Lambda]$ of the following form:
\begin{equation}
\Lambda=\bigoplus_{i=1}^n \OO e_i
\end{equation}
where $e_i\in l_i\in \FF$.  $\Ac_\FF$ is then an apartment of $\BB_e\GL(V)$, and every apartment is uniquely determined by a frame in this way. \\
\\
The group $\GL(V)$ has a natural action of $\BB_e\GL(V)$, namely the one induced from the action of $\GL(V)$ on $V$.  This action preserves distance in the building. 
\subsection{$GL(W_1)\times GL(W_2)$ acting on $\BB_e(W_1\oplus W_2)$}\label{secW1W2}
Let $V$ be a vector space over $F$.  Fix a maximal Levi subgroup $L$ of $\GL(V)$.   Associated to $L$ are subspaces $W_1,W_2\subset V$ satisfying $V=W_1\oplus W_2$.  Then $L\cong \GL(W_1)\times \GL(W_2)$.  In this section we will describe the orbits of the action of  $GL(W_1)\times GL(W_2)$ on $\BB_e\GL(V)^0$ in terms of an invariant $Q$.  Additionally we will give a representative of each orbit.\\
\\
Let $p_i$ be the projection of $V$ onto $W_i$ with respect to our given decomposition.  We will use these maps to define invariants of the vertices and then show for our action that these invariants classify all orbits. \\
\\
Let $\Lambda$ be an $\OO$-lattice.  We make the following definitions for $i=1,2$:
\begin{eqnarray}
P_i(\Lambda)&=&\Imm(p_i|_{\Lambda})\\
K_i(\Lambda)&=&\Ker(p_{i'}|_{\Lambda})=\Lambda\cap W_i
\end{eqnarray}
Where $i'=(i\text{ mod } 2)+1$.\\
\\
These are a lattices in $W_i$.  

\begin{lemma}\label{K_in_P}
$K_i(\Lambda)\subset P_i(\Lambda)$
\end{lemma}
\begin{proof}
If $v\in K_i(\Lambda)=\Lambda \cap W_i$, then $v\in \Lambda$, so $p_i(v)\in P_i(\Lambda)$.  But $p_i(v)=v$ since $v\in W_i$.
\end{proof}

Another basic lemma which will not be used immediately but will be useful later on is the following.

\begin{lemma}\label{Lam_eq}
Let $\Lambda,\Lambda'$ be $\OO$-latices, and assume that $\Lambda\subset \Lambda'$.  Furthermore, assume that $P_i(\Lambda)=P_i(\Lambda')$ and $K_i(\Lambda)=K_i(\Lambda')$.  Then $\Lambda=\Lambda'$.
\end{lemma}
\begin{proof}
Let $v'\in\Lambda'$, we wish to show $v'\in \Lambda$.  Write $p_i(v')=w_i'$.  Because $w_2'\in P_2(\Lambda)$ there is a $w_1\in W_1$ so that $w_1+w_2'\in \Lambda$.  Then $w_1'-w_1\in \Lambda'$.  Hence $w_1'-w_1\in K_1(\Lambda')=K_1(\Lambda)$, and so $w_1'-w_1\in \Lambda$.  Therefore $(w_1'-w_1)+(w_1+w_2')=v'\in\Lambda$.
\end{proof}

By lemma \ref{K_in_P} we can define $Q_i(\Lambda)=P_i(\Lambda)/K_i(\Lambda)$.  This is a finite $\OO$-module.

\begin{prop}\label{theta_map}
$Q_1(\Lambda)\cong Q_2(\Lambda)$ as $\OO$-modules.  This isomorphism class will be denoted by $Q(\Lambda)$.
\end{prop}
\begin{proof}
We make slight modifications to the proof found in \cite{Ne}.  Let $p_i':\Lambda\rightarrow Q_i(\Lambda)$ be the composition of $p_i$ with the natural projection map $\pi_i:P_i(\Lambda)\rightarrow Q_i(\Lambda)$.  We define a map so that $\forall v\in \Lambda$

\begin{equation}\label{thetamap}
\begin{array}{rll}
\Theta:Q_1(\Lambda)&\rightarrow & Q_2(\Lambda)\\
p_1'(v)&\mapsto& p_2'(v)
\end{array}
\end{equation}
 We will show that $\Theta$ is well defined, and is an isomorphism.\\
\\
Let $w_1+w_2,w_1'+w_2'\in \Lambda$ with $w_i,w_i'\in W_i$ and $\pi_1(w_1)=\pi_1(w_1')$.  Then $\pi_1(w_1-w_1')=0$, and there for $w_1-w_1'\in K_1(\Lambda)$.  Therefore $w_2-w_2'\in K_2(\Lambda)$ and $\pi_2(w_2)=\pi_2(w_2')$ showing $\Theta$ is well defined. It is an isomorphism, because the map $\theta$, defined by reversing the rolls of $1$ and $2$, is an inverse map.   
\end{proof}
We now show that $Q$ is a complete invariant of the action of $L$ on $\BB_e\GL(V)^0$.
\begin{thm}\label{Q_invar}
$\Lambda,\Lambda'$ be $\OO$-lattices.  Then $\Lambda$ and $\Lambda'$ are in the same $\GL(W_1)\times \GL(W_2)$ orbit if and only if $Q(\Lambda)=Q(\Lambda')$.
\end{thm}
\begin{proof}
$Q(\Lambda)$ is a $\GL(W_1)\times\GL(W_2)$-invariant since each factor of $\GL(W_i)$ commutes with the projection map $p_i$.  We must show that if $Q(\Lambda)=Q(\Lambda')$ then $\exists g\in \GL(W_1)\times \GL(W_2)$ so that $\Lambda=g\Lambda'$.\\
\\
By \cite{Ne} we know $\exists g_1\in GL(W_1)$ and $g_2\in GL(W_2)$ so that $g_i P_i(\Lambda')=P_i(\Lambda)$ and $g_i K_i(\Lambda')=K_i(\Lambda)$.  So we may replace $\Lambda'$ with $\Lambda''=(g_1,g_2)\Lambda'$.  Let $\Theta$ be the map from \ref{theta_map} associated to $\Lambda$, and $\Theta''$ associated to $\Lambda''$.\\
\\  
We claim $\Lambda=\Lambda''$ if and only if $\Theta=\Theta''$.  To prove this we show that one can reconstruct $\Lambda$ from $\Theta$ (which implicitly encodes $Q_i(\Lambda)$ as the domain and range of the map), by taking 
\begin{equation}
\Lambda_\Theta =\{w_1+w_2|w_i\in P_i(\Lambda)\text{ and } \Theta(\pi_1(w_1))=\pi_2(w_2)\}
\end{equation}
First we show $\Lambda\subset \Lambda_\Theta$.  Let $w=w_1+w_2\in\Lambda$, then by definition of $\Theta$ we have $\Theta (\pi_1(w_1)=\Theta(\pi_2(w_2))$.  And so $v\in \Lambda_\Theta$.  We now show $\Lambda_\Theta\subset \Lambda$.  Let $w_1+w_2\in \Lambda_\Theta$.  Then $w_1\in P_1(\Lambda)$ so there is a $w_2'\in P_2(\Lambda)$ so that $w_1+w_2'\in \Lambda\subset \Lambda_\Theta$.  Then $0+(w_2-w_2')\in \Lambda_\Theta$.  So $\pi_2(w_2-w_2')=0$ which implies $w_2-w_2'\in K_2(\Lambda)\subset \Lambda$.  Hence $w_1+w_2=(w_1+w_2')+(w_2-w_2')\in \Lambda$ as desired.\\
\\
To complete the theorem, we will show there is a $g\in \text{stab}(P_2(\Lambda))\cap \text{stab}(K_2(\Lambda))$  which takes $\Theta''$ to $\Theta$.  There is a $\overline{h}\in \GL(P_2(\Lambda)/K_2(\Lambda))$ so that $(1,\overline{h})\Theta''=\Theta$.  Let $h$ be a pull back of $\overline{h}$ to $h\in \text{stab}(P_2(\Lambda))\cap \text{stab}(K_2(\Lambda))\in \GL(W_2)$.  Then $(1,h)\Lambda''=\Lambda$.
\end{proof}

Now let $[\Lambda]\in\BB_e\GL(V)^0$, and $c\in F^\times$.  Since $Q(\Lambda)=Q(c\Lambda)$ we will abuse notation and write $Q([\Lambda])=Q(\Lambda)$.
\begin{cor}
$Q([\Lambda])$ is a complete invariant of the action of $GL(W_1)\times GL(W_2)$ on the space of vertices in $\BB_e(V)^0$.
\end{cor}
\subsection{Orbit Representatives}
We now give a set representatives of each orbit.  We first do this in the case when $V$ is 2 dimensional, and then use this case to determine representatives for higher dimensions.
\subsubsection{dim$(V)=2$}
Let $V$ be a two dimensional vector space over $F$, with decomposition $V=W_1\oplus W_2$.  Assume that $W_i$ is spanned by the vector $e_i$.  We then define the following class of lattices:
\begin{equation}\label{Lamk}
\Lambda^k=\text{span}_\OO <e_1,\varpi^{-k}e_1+e_2>
\end{equation}
\begin{prop}
$Q([\Lambda^k])\cong \OO/\PP^k$
\end{prop}
\begin{proof}
$P_1(\Lambda^k)=<\varpi^{-k}e_1>$ and $K_1(\Lambda^k)=<e_1>$.  Therefore $Q(\Lambda)\cong \PP^{-k}/\OO\cong \OO/\PP^k$.
\end{proof}
\begin{cor}\label{2dim_reps}
$\{[\Lambda^k]\}_{k=0}^\infty$ is a complete set of representatives for the action of $\GL(W_1)\times \GL(W_2)$ on $\BB_e\GL(V)^0$
\end{cor}
\begin{proof}
Let $[\Lambda]\in \BB_e\GL(V)^0$.  Then $Q([\Lambda])\cong \OO/\PP^k$ for some $k\in \Nb$.  By theorem \ref{Q_invar} $[\Lambda]$ is in the orbit of $\Lambda^k$.
\end{proof}
\subsubsection{General $V$}
We now describe representatives when $V$ is $n$ dimensional.  We may assume that $\text{dim}W_i=n_i$ and $n_1\leq n_2$.  Choose a basis $\{e_1,\ldots, e_{n_1}\}$ of $W_1$ and $\{f_1,\ldots,f_{n_2}\}$, and let $Y_i=\text{span}_F(e_i,f_i)$, for $1\leq i \leq n_1$.  Let $\alpha=(\alpha_i) \in \Nb^{n_1}$.  Let $[\Lambda^{\alpha_i}]\in \BB_e\GL(Y_i)$ defined as in equation \ref{Lamk} with respect to the basis $\{e_i,f_i\}$.  This allows us to define the following class of lattices:
\begin{equation}\label{Lamal}
\Lambda^\alpha =\bigoplus_{i=1}^{n_1} \Lambda^{\alpha_i}\bigoplus_{i=n_1+1}^{n_2} \OO f_i
\end{equation}
\begin{prop}\label{ndim_reps}
Let $A^n=\{\alpha=(\alpha_i)\in \Nb^n|  \alpha_i \geq \alpha_{i+1}\}$.  Then $[\Lambda^\alpha]_{\alpha\in A^{n_1}}$ is a complete set of representatives of the orbits of $\GL(W_1)\times \GL(W_2)$ acting on $\BB_e\GL(V)^0$.
\end{prop}
\begin{proof}
By \cite{Ne} $Q_1([\Lambda])\cong \bigoplus_{i=1}^{n_1} \OO/\PP^{\alpha_i}$ where $\alpha_i\in \Nb$.  We may assume $\alpha_i\geq \alpha_{i+1}$.  Then by theorem \ref{Q_invar} $[\Lambda]$ is in the same orbit as $[\Lambda^\alpha]$.

\end{proof}

\section{Geometric interpretation of $Q$}
\subsection{Distance Between Orbits}
The main result of section \ref{secW1W2} gives an invariant $Q$ of the action of $L=\GL(W_1)\times \GL(W_2)$ acting on $\BB_e\GL(W_1\oplus W_2)^0$.  In this section we give a geometric interpretation of this invariant in terms of a distance between orbits.\\
\\
By proposition \ref{ndim_reps} we may identify the space of orbits $L\backslash  \BB_e\GL(V)$ with $A_n$.  We define a function called the orbital distance as follows:
\begin{equation}
\begin{array}{cll}
d_O:A^n\times A^n &\rightarrow & \Nb\\
(\alpha,\beta) & \hookrightarrow & {\text{{\large max}} \atop \tiny{i=1\text{ to }n}}    (|\alpha_i-\beta_i|) 
\end{array}
\end{equation}
The main result of this section is that the name ``orbital distance'' is justified.  That is $d_O$ is actually the minimum distance between to orbits as measured in the $1$-skeleton of the building $\BB_e\GL(V)$.\\
\\
For simplicity if $[\Lambda]\in\BB_e(V)$ then let $L[\Lambda]$ denote the orbit of $[\Lambda]$ under $L$.

\begin{prop}\label{dis<1}
Let $[\Lambda_1],[\Lambda_2]\in \BB_e\GL(V)$ be incident, then $d_O(L[\Lambda_1],L[\Lambda_2])\leq 1$.
\end{prop} 
\begin{proof}
Let $[\Lambda_1],[\Lambda_2]$ be two incident vertices with $\varpi\Lambda_1\subset \Lambda_2\subset \Lambda_1$.   Let $L[\Lambda_1]$ be identified with $\alpha\in A^{n_1}$ and $L[\Lambda_2]$ with $\beta\in A^{n_1}$.  We have 
\begin{eqnarray}
\varpi P_i(\Lambda_1)\subset P_i(\Lambda_2)\subset P_i(\Lambda_1)\\
\varpi K_i(\Lambda_1)\subset K_i(\Lambda_2)\subset K_i(\Lambda_1)
\end{eqnarray}
There are two extreme cases.  First $P_1(\Lambda_2)=P_1(\Lambda_1)$ and $K_1 (\Lambda_2)=\varpi K_1(\Lambda_1)$.  In this case $\alpha_i=\beta_i+1$ for all $i$.\\
In the second case $P_1(\Lambda_2)=\varpi P_1(\Lambda_1)$, and $K_1(\Lambda_2)=K_1(\Lambda_1)\cap \varpi P_1(\Lambda_1)\supset \varpi K_1(\Lambda_1)$.  In this case $\alpha_i=\beta_i-1$ or $\alpha_i=\beta_i$.\\
The above argument shows that no mater what $P_1(\Lambda_2)$, and $K_1(\Lambda_2)$ are we have $|\alpha_i-\beta_i|\leq 1$ as desired.
\end{proof}
Proposition shows that if two incident vertices are in different orbits, then their $L$-orbits have orbital distance 1.  To show $d_O$ is actually the proposed metric we need to show if two orbits have orbital distance 1, then there are incident representatives of each orbit.  The following technical lemma proves this. 
\begin{lemma}\label{dis_dec}
Let $[\Lambda_1],[\Lambda_2]\in \BB_e\GL(V)$.  Assume $d_O(L[\Lambda_1],L[\Lambda_2])=k>0$.  Then there is an $[\Lambda_3]\in\BB_e\GL(V)$ incident to $[\Lambda_2]$ so that $d_O(L[\Lambda_1],L[\Lambda_3])=k-1$.  
\end{lemma}
\begin{proof}
Let $[\Lambda_1],[\Lambda_2]$ be as in the statement of the lemma.  Since we are working in $L$-orbits, and $L$ preserves distance in $\BB_e\GL(V)$ we may choose any representatives for $[\Lambda_1]$ and $[\Lambda_2]$ that we like.  In particular if $L[\Lambda_1],L[\Lambda_2]$ are identified with $\alpha,\beta\in A^{n_1}$ respectively, we may take for our representatives $\Lambda^\alpha,\Lambda^\beta$ respectively, as in proposition \ref{ndim_reps}. \\
\\
Recall that if $W_1$ has basis $\{e_i\}_{i=1}^{n_1}$ and $W_2$ has basis $\{f_i\}_{i=1}^{n_2}$ then $\Lambda^\alpha=\bigoplus_{i=1}^{n_1} \Lambda^{\alpha_i}\bigoplus_{i=n_1+1}^{n_2}\OO f_i$ where $\Lambda^{\alpha_i}=<e_i,\varpi^{-\alpha_i},f_i>$.  We now define a series of sublatticies $M^\alpha_i,N^\alpha_i$ which will allow us to define a lattice $\Lambda_3$ with the desired properties.  Let $M^{\alpha_i}=<e_i,\varpi{-\alpha_i+1}e_i+\varpi f_i>$ if $\alpha_i>0$, and $N^{\alpha_i}=<\varpi e_i,\varpi^{-\alpha_i}e_i+f_i>$.  We have that $\varpi\Lambda^{\alpha_i}\subset M^{\alpha_i},N^{\alpha_i}\subset \Lambda^{\alpha_i}$.\\
\\
We now calculate $Q(M^{\alpha_i})$ and $Q(N^{\alpha_i})$ with respect to $E_i=\text{span}(e_i)$ and $F_i=\text{span}(f_i)$.  $P_1(M^{\alpha_i})=<\varpi^{-\alpha_i+1}e_i>$ and $K_1(M^{\alpha_i})=<e_i>$.  So $Q(M^{\alpha_i})$ is represented by $\alpha_i-1\in A^1$.  $P_1(N^{\alpha_i})=<\varpi^{-\alpha_i}e_i>$ and $K_1(n^{\alpha_i})=<\varpi e_i>$. Hence $Q(N^{\alpha_i})$ is represented by $\alpha_i+1\in A^1$.\\
\\
We now construct $\Lambda_3$.  Let $M=\{i|\alpha_i-\beta_i=-k\}$ and $N=\{i|\alpha_i-\beta_i=k\}$ and set $S=\{1,2\ldots,n\}\backslash (M\cup N)$.  Then define $\Lambda_3$ as follows:
\begin{equation}
\Lambda_3=\bigoplus_{i\in S}\Lambda^{\beta_i}\bigoplus_{i\in M} M^{\beta_i} \bigoplus_{i\in N} N^{\beta_i}\bigoplus_{i=n_1+1}^{n_2} \OO f_i
\end{equation}
By construction we have that both $[\Lambda_2]$ and $[\Lambda_3]$ incident and $d_O(L[\Lambda_1],L[\Lambda_3])=k-1$ as desired.
\end{proof}
Together proposition \ref{dis<1} and lemma \ref{dis_dec} give us the following theorem.
\begin{thm}\label{orbit_dis}
Let $[\Lambda_1],[\Lambda_2]\in\BB_e\GL(V)^0$.  Then $d_O(L[\Lambda_1],L[\Lambda_2])$ is the minimal distance between any two representatives of the orbits as measured in the $1$-skeleton of $\BB_e\GL(V)$.
\end{thm}
Theorem \ref{orbit_dis} gives a complete combinatorial description of the geometry of the orbit space $L\BB_e\GL(V)^0$.  The following figure is the quotient space for $L\backslash\BB_e\GL(V)$ when $V$ is $4$ dimensional and $n_1=n_2=2$
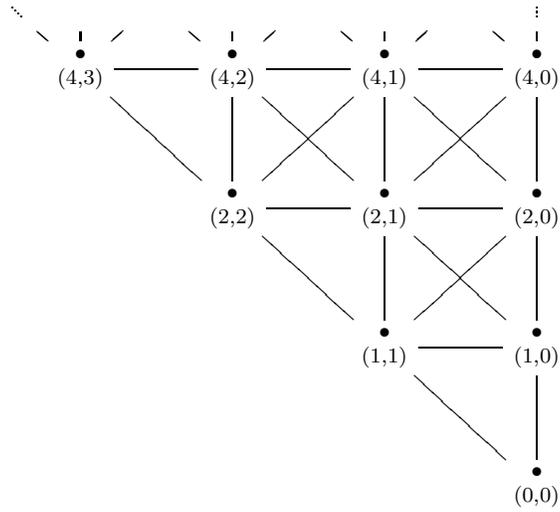
\begin{equation}
\xymatrix@=0.12cm{
 & & & & & & & & & & & & &  & \\
 &\ar@{.}[ul] & & & & & & & & & & & & &  \ar@{.}[u] \\ 
 & & \bullet\atop{(4,3)} \ar@{-}[ur] \ar@{-}[u] \ar@{-}[ul]  & & & & \bullet\atop{(4,2)}  \ar@{-}[ur] \ar@{-}[u] \ar@{-}[ul]\ar@{-}[llll] & & & & \bullet\atop{(4,1)} \ar@{-}[ur] \ar@{-}[u] \ar@{-}[ul]\ar@{-}[llll] & & & & \bullet\atop{(4,0)} \ar@{-}[u] \ar@{-}[ul]\ar@{-}[llll]\\
& & & &  & & & &  & & & &  & & \\
& & & &  & & & &  & & & &  & & \\
& & & &  & & & &  & & & &  & & \\
& &    & & & & \bullet\atop{(2,2)} \ar@{-}[uuuu] \ar@{-}[uuuullll]\ar@{-}[uuuurrrr] & & & & \bullet\atop{(2,1)} \ar@{-}[uuuu] \ar@{-}[uuuullll] \ar@{-}[llll] \ar@{-}[uuuurrrr] & & & & \bullet\atop{(2,0)} \ar@{-}[uuuu] \ar@{-}[uuuullll] \ar@{-}[llll]  \\
& & & &  & & & &  & & & &  & &  \\
& & & &  & & & &  & & & &  & & \\
& & & &  & & & &  & & & &  & & \\
  & &  & & & &  & & & & \bullet\atop{(1,1)} \ar@{-}[uuuu] \ar@{-}[uuuullll] \ar@{-}[uuuurrrr]& & & & \bullet\atop{(1,0)} \ar@{-}[uuuu] \ar@{-}[uuuullll] \ar@{-}[llll] \\
& & & &  & & & &  & & & &  & & \\
& & & &  & & & &  & & & &  & & \\
& & & &  & & & &  & & & &  & & \\
& &  & & & &  & & & &  & & & & \bullet\atop{(0,0)} \ar@{-}[uuuu] \ar@{-}[uuuullll] \\
}
\end{equation}
\newpage
\subsection{Distance to $\overline{\Ac_{\FF_1\cup \FF_2}}$ in $\BB_e(W\oplus W)$}
There is an important special case of theorem \ref{orbit_dis}.  The orbit for which $Q(\Lambda)=0$ is distinguished.  In this section we give both a description of this orbit, as well as another description of this distance from a given point to this orbit.\\
\\
Recall from section \ref{secBe} that an apartment $\Ac_\FF$ is specified by a frame $\FF$ in $W_1\oplus W_2$.  Denote by $\text{Frame}(V)$ the set of all frames in a vector space $V$.  We will be interested in the following collection of apartments:
\begin{equation}
\AF=\bigcup_{\FF_1\in\text{Frame}(W_1)\atop \FF_2\in \text{Frame}(W_2)}\Ac_{\FF_1\cup \FF_2}
\end{equation}
\begin{prop}
$\AF$ is a subbuilding of $\BB_e\GL(V)$.
\end{prop}
\begin{proof}
Since $\AF$ is a union of apartments from an actual building all that needs to be shown is that any two chambers $C_1,C_2\in \AF$ are in a common apartment.  Let $\Lambda_1\supset \Lambda_2\supset\ldots \supset\Lambda_{n}\supset \varpi\Lambda_1$ be a chain of $\OO$-lattices corresponding to a chamber $C\in \AF$, and $M_1\supset M_2 \supset \ldots \supset M_{n}\supset \varpi M_1$ a chain of lattices corresponding to a chamber $D\in \AF$.  Since each $[\Lambda_i]\in\AF$ we can write  $\Lambda_i=\Lambda_i^1\oplus \Lambda^2_i$ with $[\Lambda_i^j]\in \BB_e(W_j)$.  Similarly for the $M_i$.  The $\{[\Lambda_i^j\}_{i=1}^{n}$,$\{M_i^j\}_{i=1}^{n}$ specify facets $C_j,D_j\in \BB_e(W_j)$  since $\Lambda_1^j\supset \Lambda_i^j\supset \varpi\Lambda_1^j$ (it will be the case that some of the $\Lambda_i^j=\Lambda_{i+1}^j$ but this will not matter), and similarly for the $M_i^j$.  Then there are common apartments $\Ac_j\subset \BB_e\GL(W_j)$ which contain $C_j$ and $D_j$.  Since each $\Ac$ is specified by a frame $\FF_j$ in $W_j$.  Then $\Ac_{\FF_1\cup\FF_2}$, the apartment specified by $\FF_1\cup \FF_2$, contains the chambers $C$ and $D$.
\end{proof}

Now let $[\Lambda]\in \BB_e\GL(V)^0$. We define a function on $\BB_e\GL(V)^0$ as follows:
\begin{equation}
\begin{array}{cll}
d_\Ac:\BB_e(W_1\oplus W_2)&\rightarrow & \Nb\\
$[$\Lambda$]$ & \mapsto & \log_\PP[\text{Ann}(Q(\Lambda))]
\end{array}
\end{equation}
\begin{thm}
Let $[\Lambda] \in \BB_e\GL(V)^0$ then $d_\Ac([\Lambda])=d_O(L[\Lambda],\AF )$.
\end{thm}
\begin{proof}
This follows from theorem \ref{orbit_dis}, and the fact that $\AF$ is associated to $(0)\in A^n$.
\end{proof}
In the special case when $n=1$ $\AF$ is just an apartment of $\BB_e\GL(V)^0$.   Then $d_\Ac$ is just measuring the distance of a given point to a fixed apartment.  This suggests that one may be able to find the distance of a vertex to a fixed apartment by studying the action of a maximal split torus on the building.

\bibliography{bib}{}
\bibliographystyle{plain}
\end{document}